\documentclass[11pt]{amsart}
\usepackage{graphicx}
\usepackage{latexsym}
\usepackage{amssymb,amsmath}
\usepackage{amsfonts}
\usepackage{amscd}
\usepackage{a4wide}
\usepackage{epsfig}
\usepackage{amsmath, amssymb, amsthm, amsfonts, amsgen} 
\usepackage{stmaryrd}
\usepackage{mathrsfs}
\usepackage[foot]{amsaddr}
\usepackage[centertags]{amsmath}
\usepackage{indentfirst}
\usepackage{fancyhdr}
\usepackage{bm}
\usepackage{epstopdf}

\graphicspath{ {./} }

\numberwithin{equation}{section}
\newtheorem{definition}{Definition}[section]
\newtheorem{lemma}[definition]{Lemma}
\newtheorem{theorem}[definition]{Theorem}
\newtheorem{proposition}[definition]{Proposition}
\newtheorem{corollary}[definition]{Corollary}
\newtheorem{remark}[definition]{Remark}
\newtheorem{conjecture}[definition]{Conjecture}

\let \l =\lambda

\newcommand{\R}{\mathbb{R}}

\newcommand{\om}{\Omega}

\newcommand{\cof}{{\rm cof}\,}

\newcommand{\ul}{u_{\lambda}}
\newcommand{\1}{{\bf 1}}
\newcommand{\tr}{{\rm tr}\,}
\newcommand{\eps}{\epsilon}
\newcommand{\sca}{\mathcal{A}}

\newcommand{\rb}{\bar{\rho}}

\newcommand{\scf}{\mathcal{F}}

\renewcommand{\rb}{\raisebox{2.7ex}{}}

\begin{document}
\author[J.J. Bevan]{Jonathan J. Bevan}
\address[J.J. Bevan]{Department of Mathematics, University of Surrey, Guildford, GU2 7XH, United Kingdom.  (Corresponding author: \textbf{t:} $+44 (0)1483 \  682620$.)}
\email[Corresponding author]{j.bevan@surrey.ac.uk}
\author[J.H.B. Deane]{Jonathan H.B. Deane}
\address[J.H.B. Deane]{Department of Mathematics, University of Surrey, Guildford, GU2 7XH, United Kingdom.}
\email{j.deane@surrey.ac.uk}

\title[A calibration method for estimating critical cavitation loads]
{A calibration method for estimating critical cavitation loads from below in 3D nonlinear elasticity}

\begin{abstract}  In this paper we give an explicit sufficient condition for the affine map $u_\lambda(x):=\lambda x$ to be the global energy minimizer of a general class of elastic stored-energy functionals $I(u)=\int_{\Omega} W(\nabla u)\,dx$ in three space dimensions, where $W$ is a polyconvex function of $3 \times 3$ matrices.  The function space setting is such that cavitating (i.e., discontinuous) deformations are admissible.  In the language of the calculus of variations, the condition ensures the quasiconvexity of $I(\cdot)$ at $\lambda \1$, where $\1$ is the $3 \times 3$ identity matrix.  Our approach relies on arguments involving null Lagrangians (in this case, affine combinations of the minors of $3 \times 3$ matrices), on the previous work \cite{BZ15}, and on a careful numerical treatment to make the calculation of certain constants tractable.   We also derive a new condition, which seems to depend heavily on the smallest singular value $\lambda_1(\nabla u)$ of a competing deformation $u$, that is necessary for the inequality $I(u) < I(\ul)$, and which, in particular, does not exclude the possibility of cavitation.  
\end{abstract}

\maketitle

\section{Introduction}

In this paper we consider an established model of elastic material that is capable of describing cavitation, that is, of admitting energy minimizers that are discontinuous.    This phenomenon was first analysed in the setting of hyperelasticity by Ball in \cite{Ba82};  since then, a large and sophisticated literature has developed, including but not limited to \cite{St84,Si86,MS95,SS08,SS08prime,HMC10,HMC11},  part of which focuses on finding boundary conditions which, when obeyed by all competing deformations, ensure that cavitation does \emph{not} occur.   It is to the latter body of work that we contribute by considering the case of purely bulk energy
\begin{align*} I(u) & =\int_{\om} W(\nabla u(x))\,dx,
\end{align*}
where $u: \om \to \R^3$ represents a deformation of an elastic material occupying the domain $\om$ in a reference configuration, and where $W$ is a suitable stored-energy function.   In the three dimensional setting, we give an explicit characterization of those affine boundary conditions of the form
\begin{align}\label{qc}  \ul(x)=\l x,
\end{align}
where $\l>0$ is a parameter, such that the quasiconvexity inequality
\begin{align*} I(u) \geq I(\ul)
\end{align*}
holds among all suitable maps $u$ agreeing with $\ul$ on $\partial \om$.   It is by now well established that if $\l$ is large enough, $\l \geq \l_{\textrm{crit}}$, say,  then such an inequality cannot hold.   Thus we probe $\l_{\textrm{crit}}$ by finding $\l_0$ such that \eqref{qc} holds whenever $\l \leq \l_0$.  This question has been addressed in \cite{MSS96} and, more recently, in \cite{BZ15}.   In this paper we use a new approach, involving the addition of a suitable null Lagrangian (a method sometimes known as calibration), to deduce concrete lower bounds on $\lambda_{\textrm{crit}}$ in the three dimensional case.    


The analysis centres ostensibly on functions of the singular values of $3 \times 3$ matrices.  Let $A$ be a $3 \times 3$ matrix.  Then the singular values of $A$ are normally written as $\l_j(A)$, for $j=1,2,3$, and their squares are the eigenvalues of $A^TA$.   See \cite[Chapter 13]{Dac08} or \cite[Section 3.2]{Ci04} for useful introductions to singular values, as well as \cite{Ba77,Ba82,Si86} for an illustration of their use in nonlinear elasticity.  Singular values arise naturally in the stored-energy functions of isotropic elastic materials, and also in lower bounds which can be derived from them.  Such was the case in \cite{BZ15}, where, for $2 < q < 3$ and for convex functions $Z$ and $h$,  a stored-energy function very similar to\footnote{The original functional contained an `artificial' quadratic term $r|A|^2$, with $r$ large, to deal with the difficulties presented by the function $P$. This is no longer needed thanks to the calibration method we introduce.}
\begin{align*} W(A) = |A|^q + Z(\cof A) + h(\det A)
\end{align*}
was shown to obey the inequality
\begin{align}\label{moosewood} I(u)-I(\ul) & \geq \int_{\om} \kappa |\nabla u - \nabla \ul|^q + h'(\l^3)\Pi_{j=1}^{3}(\l_i(\nabla u)-\l)+\l P(\nabla u)\,dx.\end{align}
The function $P$ is defined by
\begin{equation}\label{f1}
P(A) = \sum_{1 \leq i < j \leq 3} \lambda_{i}(A)\lambda_{j}(A) - \lambda \sum_{1 \leq i \leq 3} \lambda_{i}(A)
\end{equation}
and the constant $\kappa$ satisfies bounds defined in \eqref{j:kappalimits} below.  By grouping the first two integrands in \eqref{moosewood} together, it is possible to find conditions on $\l$ such that $\int_{\om} \kappa |\nabla u - \nabla \ul|^q + h'(\l^3)\Pi_{j=1}^{3}(\l_i(\nabla u)-\l) \,dx \geq 0$.  However, the corresponding inequality for $P$, namely 
\begin{align}\label{coruscate}\int_{\om} P(\nabla u) \,dx \geq 0,
\end{align}
which, since $P(\l \1)=0$, is equivalent to the quasiconvexity of $P$ at $\l \1$, remains an open question.  We show in this paper that $P$ does satisfy a condition necessary for quasiconvexity at $\l \1$ (see Proposition \ref{p1}, part (a): rank-one convexity at $\l \1$), but that the most tractable sufficient condition for \eqref{coruscate} cannot hold (see Proposition \ref{p1}, part (b): polyconvexity at $\l \1$)).  Trying instead to find conditions under which each of 
\begin{align} \nonumber \int_{\om} (\kappa/2) |\nabla u - \nabla \ul|^q + h'(\l^3)\Pi_{j=1}^{3}(\l_i(\nabla u)-\l)\,dx & \geq 0 \\
\label{snakebark}\int_{\om} (\kappa/2) |\nabla u - \nabla \ul|^q + h'(\l^3)P(\nabla u) \,dx & \geq 0
\end{align}
holds is closer to the right approach, although for reasons connected with the curvature of $P$ at $\l \1$, \eqref{snakebark} is still not possible!  This is what leads us to introduce the null Lagrangian 
\begin{align*}N(A)=\tr \cof A - \l \tr A,\end{align*}
which has the property that $\int_{\om} N(\nabla u) \,dx = 0$ for any admissible $u$ and is such that there \emph{are} conditions on $\l$ under which 
\begin{align*}\int_{\om} (\kappa/2) |\nabla u - \nabla \ul|^q + h'(\l^3)(P(\nabla u)-N(\nabla u)) \,dx & \geq 0
\end{align*}
for all admissible $u$.   See Theorem \ref{printingpress} and \eqref{estsuper} in particular.    In fact, $N$ is the unique null Lagrangian for which this method works:  see Proposition \ref{proptan}.
More generally, we remark that $P$ and $G:=P-N$ possess properties that are both interesting in their own right and, at the same time, highly non-trivial to derive. (See Section \ref{earedwillow}.)   A useful introduction to null Lagrangians can be found in \cite{BCO}.

The upper bound $\l_0$ given in the right-hand side of \eqref{estsuper} is investigated in Section \ref{plotselm} using a careful mixture of analysis and numerical techniques.   The partnership between these approaches seems to be particularly fruitful when applied to $G$ and to functions derived from it.  Accordingly, we find an explicit constant $\nu_1 \approx 0.4501$ such that if 
\begin{align*} 0 \leq \l^{3-q} h'(\l^3)& \leq \frac{\kappa}{2} (\sqrt{2})^{q-3} \nu_1^{2-q}, \end{align*}
 then $I(u) \geq I(\ul)$.  See Section \ref{plotselm}, Subsection \ref{concrete} and the appendices for details.

In Section \ref{spectral}, a careful analysis of the function 
\begin{align*} H(A)& := \Pi_{j=1}^{3}(\l_j(\nabla u)-\l) + \l G(A)\end{align*}
yields, among other things, what we believe to be new necessary conditions for the inequality $I(u) \leq I(\ul)$.  A distinguished role seems to be played by the smallest singular value, $\l_1(\nabla u)$:  see Proposition \ref{lime} in particular.

\subsection{Notation} 
The inner product between two matrices $A$ and $B$ is given by $A \cdot B = \tr A^T B$, and, as usual, $\tr A$ denotes the trace of $A$.   For a function $f: \R^{3 \times 3} \to \R$ and any $3 \times 3$ matrix $U$, the shorthand 
\begin{align*} D_{U}f(A) & =  \nabla f(A) \cdot U  \\
D^{2}_{U}f(A) & = \nabla^{2}f(A) [U,U]
\end{align*}
will be used, where as usual $\nabla^{2}f(A)[U,U]=f_{,_{(ij)(kl)}}(A)U_{ij}U_{kl}$ with the summation convention in force.   When discussing polyconvexity, which is defined when it next features in the paper, we use the shorthand notation $\R^{19}$ for the set $\R^{3 \times 3} \times \R^{3 \times 3} \times \R$ containing the list of minors $R(A):=(A, \cof A, \det A)$ of any $3 \times 3$ matrix.   The set of $3 \times 3$ square, orthogonal matrices is denoted by $O(3)$, and the subset of $O(3)$ consisting of those matrices with determinant equal to $1$ will be written $SO(3)$.  For any two vectors $a$ and $n$ in $\R^3$, the notation $a \otimes n$ will denote the matrix of rank one whose $(i,j)$ entry is $a_i n_j$.    Our notation for Sobolev spaces is standard.
 
\section{Calibration and the function $G(A)$}\label{earedwillow}

In this section we give some properties of the function $P$ and use them to derive the null Lagrangian $N$ alluded to above.   To start with, two technical results are required.

\begin{lemma}\label{motet} Let $\lambda >0$.  Then 
\begin{itemize}\item[(i)] $\sum_{i=1}^{3} D_{U}\lambda_{i}(\lambda \1) = \tr U$;\\
\item[(ii)] $\sum_{1 \leq i < j \leq 3} D_{U}\lambda_i(\lambda \1)D_{U}\lambda_j(\lambda \1) =  \frac{(\tr U)^2}{2}-\frac{|U|^2}{4}-\frac{\tr (U^2)}{4}$;\\
\item[(iii)] $\lambda \sum_{i=1}^{3} D_{U}^{2} \lambda_i (\lambda \1) + \sum_{i=1}^{3} (D_{U}\lambda_{i}(\lambda \1))^2 = |U|^2$. 
\end{itemize}
In particular, 
\begin{align}\label{d2l}\sum_{i=1}^{3} \lambda D^{2}_{U}\lambda_{i}(\lambda \1) & = \frac{|U|^2 - \tr (U^2)}{2}\end{align}
and 
\begin{align}\label{sumsq}\sum_{i=1}^{3} (D_{U}\lambda_{i}(\lambda \1))^2 = \frac{|U|^2 + \tr (U^2)}{2}.\end{align}
\end{lemma} 
\begin{proof} \textbf{Parts (i) and (iii):}  Let $\lambda_i (\l \1 + hU) = \l + D_{U}\l_i(\l \1) + \frac{h^2}{2}D_U^2\l_i(\l \1) + o(h^2)$ for each $i$ and insert into the identity $\sum_{i=1}^{3} \l_i(\l \1 + h U) = |\l \1 + h U|^2$.  Part (i) follows by comparing terms of order $h$ and part (iii) by comparing terms of order $h^2$. 

\vspace{2mm}
\noindent \textbf{Part (ii)}:  Let $A=\l\1+hU$ and note that, by definition, each $\l_i(A)$ is a root $z_i$, say, of the polynomial $\det(A^T A - z^2 \1) = 0$.   Now 
$A^T A = \l^2 \1+ 2 \l h U_s+  h^2 U^T U$, so 
\begin{align}\nonumber 0  & = \det( (\l^2 - z^2)\1 + 2 \l h U_s + h^2 U^{T} U) \\
\label{detexpand}& = \tau^3 + \tau^2 \tr (2 \l h U_s + h^2 U^{T} U) + \tau \tr \cof (2 \l h U_s + h^2 U^{T} U)  + \det(2 \l h U_s + h^2 U^{T} U), 
\end{align}
where $U_s:=(U+U^T)/2$ is the symmetric part of $U$ and $\tau:=\l^2 - z^2$.  Using the development of $\l_i(\l \1 + h U)$ given above, but this time keeping only terms of order $h$, it follows that $\tau = - 2 h \l D_U \l_i(\l \1) + o(h^2)$.  Putting this into \eqref{detexpand} and writing  $Z_i:= D_U \l_i(\l \1)$ for brevity,
shows that the $Z_i$ are roots of the following polynomial equation:
\begin{align*}-8 \l^3 h^3 Z^3 + 8 \l^3 h^3 \tr U Z^2 - 8 \l^3 h^3 \tr \cof U_s Z + 8 \l^3 h^3 \det U_s + o(h^3) & = 0.
\end{align*}
Dividing by $-8 \l^3 h^3$, letting $h \to 0$  and using the identity
\[ \tr \cof U_s = \frac{(\tr U)^2}{2} - \frac{|U|^2}{4} - \frac{ \tr U^2}{4}\]
gives
\[Z^3 - (\tr U) Z^2 +  \left(\frac{(\tr U)^2}{2} - \frac{|U|^2}{4} - \frac{ \tr U^2}{4}\right) Z - \det U_s = 0.\]
The roots $Z_1,Z_2,Z_3$ must therefore satisfy 
\begin{align}\label{sumz} - \sum_{i=1}^{3} Z_i & = - \tr U, \\
\label{sumzz} \sum_{1 \leq i < j \leq 3} Z_i Z_j &  = \frac{(\tr U)^2}{2} - \frac{|U|^2}{4} - \frac{ \tr U^2}{4} .    
\end{align}
Replacing each $Z_i$ with $ D_U \l_i(\l \1)$ in equation \eqref{sumz} merely recovers (or provides an alternative derivation of) part (i) of the lemma, while equation \eqref{sumzz} delivers part (ii).   

Equation \eqref{sumsq} now follows by using the identity
\begin{align*} \sum_{i=1}^3 Z_i^2 = \left( \sum_{i=1}^{3} Z_i \right)^2 - 2 \sum_{1 \leq i < j \leq 3} Z_i Z_j 
\end{align*}
together with parts (i) and (ii) above.  Finally, \eqref{d2l} follows from (iii) above and \eqref{sumsq}.  This concludes the proof.
\end{proof}

Note that \eqref{d2l} tells us, via the Cauchy-Schwarz inequality, that $D^2_{U}(\sum_{i=1}^{3} \lambda_{i})(\lambda \1)$ vanishes if and only if $U$ is a symmetric matrix.   Moreover, \eqref{sumsq} implies that $\sum_{i=1}^3 (D_U\l_i(\l\1))^2$ vanishes if and only if $U$ is antisymmetric, and that in this case $D_U\l_i(\l\1)=0$ for each index $i$.

\begin{proposition}\label{proptan}  Let $C_1$ and $C_2$ be fixed $3 \times 3$ matrices, let $C_3$ be a real number, and let 
\begin{align*}N(A) = C_1 \cdot (A- \lambda \1) + C_2 \cdot (\cof A- \lambda^2 \1) + C_3(\det A - \l^3) \end{align*}
for all $A \in \R^{3 \times 3}$.  Then 
\begin{align} \label{partialpc} D_{U}P(\lambda \1) & = D_{U}N(\lambda \1) \quad \quad \forall U \in \R^{3 \times 3}
\end{align}
if and only if $C_1, C_2$ and $C_3$ are related by the equation
\begin{align}\label{hoot} (\l  - \l^2 C_3 - \l \tr C_2) \1 & = C_1 - \l C_{2}^T.
\end{align}
Moreover, \begin{align} \label{partialpc2} D^2_{U}P(\lambda \1) & = D^2_{U}N(\lambda \1) \quad \quad \forall U \in \R^{3 \times 3}
\end{align}
if and only if  $C_2$ and $C_3$ are related by the equation
\begin{align}\label{hootsmon} C_2 & = (1-\lambda C_3) \1.
\end{align}
In particular, the unique quadratic null Lagrangian $N$ satisfying both \eqref{hoot} and \eqref{hootsmon} is 
\begin{align}\label{defN} N(A) = \tr \cof A - \lambda \tr A,\end{align}
and it satisfies
\begin{align}\label{nullproperty}\int_{\om}N(\nabla u)\,dx =0 \end{align}
for all $u$ belonging to $W^{1,2}(\om,\R^3)$ such that $u=\ul$ on $\partial \om$ (in the sense of trace). 
\end{proposition}

\begin{proof}
Let $A = \lambda \1 + h U$ and note that 
\begin{align*} N(A) & = \left[ U \cdot C_1 + \l \tr U \tr C_2 - \l \tr (U C_2) + \l^2 C_3 \tr U \right] h  + \left[C_2 \cdot \cof U + \l C_3 \tr \cof U \right] h^2 + h^3 \det U.\end{align*}
Next, rewrite 
\begin{align*} P(A) & = \frac{1}{2}\left( \left(\sum_{i=1}^{3} \lambda_i(A)\right)-\lambda\right)^{2} - \frac{\lambda^2}{2} - \frac{|A|^2}{2} 
\end{align*}
and, for sufficiently small $h$, write 
\begin{align*}\lambda_{i}(A)&  = \lambda + h D_{U}\lambda_{i}(\lambda \1) + \frac{h^2}{2}D^{2}_{U}\lambda_{i}(\lambda \1) + \rho_{i}(\lambda,h,U),
\end{align*}
where $\rho_{i}$ is $o(h^2)$ as $h \to 0$.    A short calculation then yields
\begin{align*} P(A) & = \lambda h \tr U + \frac{\lambda h^2}{2} \sum_{i=1}^{3} D^{2}_{U}\lambda_{i}(\lambda \1) + h^2 \sum_{1 \leq i < j \leq 3}D_{U}\lambda_i(\lambda \1)D_{U}\lambda_j(\lambda \1) + \rho \\
& = \lambda h \tr U + \frac{h^2}{4}\left(|U|^2 - \tr(U^2)\right)+\frac{h^2}{4} \left(2(\tr U)^2 -|U|^2-\tr(U^2)\right) + \rho \\
& = \lambda h \tr U + \frac{h^2}{2}\left((\tr U)^2 - \tr (U^2)\right) + \rho.\\
& = \lambda h \tr U + h^2 \tr \cof U + \rho.
\end{align*}
Here, we have used the identity $\tr \cof U = \frac{1}{2}\left((\tr U)^2 - \tr (U^2)\right)$, \eqref{d2l} and Lemma \ref{motet}(ii).  The term $\rho=\rho(\lambda,h,U)$ is $o(h^2)$ as $h \to 0$.     Comparing terms of order $h$ in this expression with the expansion for $N(A)$ given above, we see that $D_{U}P(\l 1) = D_{U}N(\l 1)$ for all $U$ if and only if \eqref{hoot} holds.   To prove the equivalence of \eqref{partialpc2} and \eqref{hootsmon}, simply compare terms of order $h^2$ to obtain
\begin{align}\label{youllhavehadyourtea}(C_2 + \l C_3 \1) \cdot \cof U = \1 \cdot \cof U\end{align}  
for all $U$, and then pick $U$ such that $\cof U= e_i \otimes e_j$.  It is then clear that \eqref{youllhavehadyourtea} is equivalent to \eqref{hootsmon}.  

Finally, to prove that $N(A)=\tr \cof A - \l \tr A$ is the unique, quadratic null Lagrangian satisfying \eqref{hoot} and \eqref{hootsmon} take $C_3=0$ in \eqref{hootsmon} and \eqref{hoot}.  The former gives $C_2=\1$, and the latter $C_1=-\l \1$, which together imply \eqref{defN}.    Equation \eqref{nullproperty} is a standard result about null Lagrangians;  to see it without recourse to general theory, simply observe that, for sufficiently smooth $\varphi$, $N(\nabla \varphi)$ can be written as a divergence.  The result then follows from the Green's theorem and an approximation argument. (The argument given in \cite[Lemma 5.5 (ii)]{Dac08} serves as a useful template.) 
\end{proof}

We remark that this establishes a simple pattern:   $N(A)$ can apparently be obtained from $P$ by noting that if $A=\textrm{Diag}\,(\lambda_1,\lambda_2,\lambda_3)$ then $P(A) = \tr \cof A - \lambda \tr A$.   

As was pointed out in the introduction, and originally conjectured in \cite{BZ15}, it would be very useful if $P$ were quasiconvex at the matrix $\l \1$.   
Our results in this direction are somewhat mixed.  We find that $P$ satisfies a condition necessary for quasiconvexity at $\l \1$, but that it does not satisfy a tractable condition sufficient condition for quasiconvexity at $\l \1$.  To be precise, (a) $P$ is rank-one convex at $\l \1$ but (b) $P$ is not polyconvex at that point.   These concepts are explained in more detail below.   We note, incidentally, that $P$ is not globally rank-one convex.   The latter is relatively easy to see:  one can immediately calculate that, for any rank-one matrix $A=a\otimes n$, $\l_1(A)=\l_2(A)=0$ and $\l_3(A)=|A|$.  In particular, $P(A) = -|A|$, which is a concave function of $A$.    The foregoing discussion is summarised in the result below.

\begin{proposition}\label{p1} Let the function $P$ be defined by \eqref{f1}.  Then 
\begin{itemize}
\item[(a)] $\l \1$ is a point of rank-one convexity of $P$, but
\item[(b)] $P$ is not polyconvex at $\lambda \1$.
\end{itemize}
\end{proposition}
\begin{proof} \noindent\textbf{(a)}:  To show (a) we only need to verify that $P(\l \1+ a \otimes n)$ is convex as a function of the rank-one matrix $a \otimes n$.  Without loss of generality, we may choose coordinates such that $n=e_1$. A calculation then shows that, if the component of $a$ in the $e_1$ direction is $a_1$, the following expression holds:
\begin{align*} P(\l \1 + a \otimes n) & = \l (|\l+a_1|-\l).
\end{align*}
This is clearly convex in $a \otimes n$, which proves part (a) of the lemma.

\vspace{1mm}
\noindent\textbf{(b)} Assume for a contradiction that $\lambda \1$ is a point of polyconvexity of $P$.   This means that there is some point $(C_1,C_2,C_3) $ in $\R^{19}$ such that
\begin{align}\label{pc_at_R_lambda} P(A) \geq P(\lambda \1) + C_1 \cdot (A - \lambda \1) + C_2 \cdot (\cof A - \lambda^2 \1) + C_3(\det A - \lambda^3)
\end{align}
for all $A$ in $\R^{3 \times 3}$.  Note that $C_3$ has to be zero because $P$ is at most quadratic.  Next, take $A$ to be a rank-one matrix such that $C_1 \cdot A=\textrm{sign}\;((C_1)_{ij})t$ for a given pair $i,j$, where $t$ is a positive parameter to be chosen shortly.   Recall that, when $A$ is a rank-one matrix, $P(A)=-|A|$.   If $(C_1)_{ij} \neq 0$, this gives
\begin{align*} -t & \geq |(C_1)_{ij}|t - \l \tr C_1 -\l^2 \tr C_2,
\end{align*}
which is easily contradicted by taking $t$ to be sufficiently large.   Therefore $C_1 = 0,$ leaving
\begin{align}\label{reduced1}P(A) & \geq P(\l \1) + C_2 \cdot (\cof A - \l^2 \1).
\end{align}
By considering $A=\l \1 + h U$ for arbitrary $U$ in $\R^{3 \times 3}$ and small $h$, it is straightforward to show that this implies $D_U P(\l \1) = D_U \bar{N}(\l \1)$,  where 
\begin{align}\label{chizlowska} \bar{N}(A) & = C_2 \cdot (\cof A - \l^2 \1). 
\end{align}
In the course of Proposition \ref{proptan} it is shown that $D_U P(\l \1)=\l \tr U$, so \eqref{chizlowska} implies
\begin{align*}
\l \tr U & = \l (\tr C_2 \tr U - \tr (U C_2))\end{align*}
for all $U$.  Rearranging this gives
\[ C_2 = (\tr C_2 - 1) \1, \]
so that $C_2=(1/2)\1$.    Putting this into \eqref{reduced1} gives
\[ P(A) \geq \frac{\tr \cof A - 3 \l^2}{2},\]
 which is easily contradicted by taking $A$ to be of rank $1$, applying the observation that $P(A)=-|A|$ for such $A$, and letting $|A| \to \infty$.   This concludes the proof.
\end{proof}


Next, with $N$ as in Proposition \eqref{defN}, we define
\begin{align}\label{defG} G(A) := P(A) - N(A)\end{align}
for all $3 \times 3$ matrices $A$.  We know by equation \eqref{partialpc} in Proposition \ref{proptan} that $P$ and $N$ are tangent at $\lambda \1$, so clearly $D_{U}G(\lambda \1) = 0$ for all $U$.  Moreover, by \eqref{hootsmon},  we also have that $D^2_{U}G(\l \1)=0$ for all $U$.
Thus $G$ behaves like $|A-\lambda \1|^3$ in a neighbourhood of $\lambda \1$, and this is a key feature which enables us to find new lower bounds for $\l_\textrm{crit}$.  The technique for doing so is described in the next section.  We also record the following useful property of $G$, which flows directly from \eqref{nullproperty}:
\begin{align}\label{intpig}
\int_{\om}P(\nabla u)\,dx = \int_{\om} G(\nabla u)\,dx 
\end{align} 
for all $u$ belonging to $W^{1,2}(\om,\R^3)$ such that $u=\ul$ on $\partial \om$ (in the sense of trace).

\section{New lower bounds on $\lambda_{\textrm{crit}}$}\label{plotselm}

Let the stored-energy function $W: \R^{3 \times 3} \to [0,+\infty]$ be given by 
\begin{align}\label{w}W(A) &= |A|^{q} + Z(\cof A) + h(\det A) 
\end{align} 
where $Z: \R^{3 \times 3} \to [0,+\infty)$ is convex and $h: \R \to [0,+\infty]$ has the following properties:
\begin{itemize}
\item[(H1)] $h$ is convex and $C^1$ on $(0,+\infty)$;
\item[(H2)] $\lim_{t \to 0+}h(t) = +\infty$ and $\liminf_{t \to +\infty}\frac{h(t)}{t} > 0$;
\item[(H3)] $h(t)=+\infty$ if $t \leq 0$.
\end{itemize}
The exponent $q$ satisfies $2 < q < 3$.  Let 
\begin{align*}
I(u) = \int_{\om} W(\nabla u) \, dx 
\end{align*} 
and define the class of admissible maps as
\begin{align*}\sca_{\l}=\{u \in W^{1,q}(\om;\R^3): \ u = \ul \ \textrm{on} \ \partial \om, \   I(u) < +\infty\}.
\end{align*}
The following argument is straightforward and can be found in \cite[Section 3]{BZ15}.  We include it here both for completeness and as a means of deriving the function $P$ defined by \eqref{f1}.  Applying \cite[Lemma A.1]{MSS96} to $A \mapsto |A|^q$ gives
\begin{equation}\label{c:q}
|\nabla u|^q \geq |\l\1|^q+q|\l\1|^{q-2} \l\1\cdot (\nabla u -\l\1)+\kappa |\nabla u-\l\1|^q,
\end{equation}
where
\begin{equation}\label{j:kappalimits}
2^{2-q} \leq \kappa \leq  q2^{1-q}  .\end{equation} 
Therefore, by \eqref{c:q} and by appealing to the convexity of $Z$ and $h$, we obtain
\begin{eqnarray}\nonumber 
W(\nabla u) &\geq & W(\nabla u_{\lambda}) + q|\l \1|^{q-2}\l \1\cdot(\nabla u    - \l \1) + \kappa |\nabla u-\l\1|^{q}
\\\label{ineq0:3d}
&+& 2\gamma \l \1 \cdot (\nabla u    - \l \1) + \gamma |\nabla u - \l \1|^{2}
\\\nonumber
&+& D_{A}Z(\cof \l \1) \cdot (\cof \nabla u - \cof \l \1)\\ \nonumber
&+& h'(\l^{3})(\det \nabla u-\l^3), 
\end{eqnarray}
for any $u \in \sca_{\l}$.   Integrating \eqref{ineq0:3d} and using the facts that both $\nabla u$ and $\cof \nabla u$ are null Lagrangians in 
$W^{1,q}(\om,\R^{3})$ for $q \geq 2$, we obtain
\begin{align}\nonumber I(u) - I(\ul) & \geq \int_{\om} \kappa |\nabla u - \lambda \1|^q + h'(\lambda^3)\lambda_1 \lambda_2 \lambda_3 \,dx  \\ \label{balladeno2} & = 
\int_{\om}   \kappa |\nabla u - \lambda \1|^q + h'(\lambda^3)\hat{\lambda}_1 \hat{\lambda}_2 \hat{\lambda}_3 \,dx + \lambda h'(\lambda^3) \int_{\om} P(\nabla u) \,dx.
\end{align}
In deriving this, it may help to recall the identity
\begin{align}\label{greenalder} \lambda_1 \lambda_2 \lambda_3 = \hat{\lambda}_{1}\hat{\lambda}_2 \hat{\lambda}_{3} + \lambda \sum_{1 \leq i < j \leq 3} \lambda_i \lambda_j - \lambda^2 \sum_{i=1}^{3} \lambda_i + \l^3,
\end{align}
where the notation $\lambda_i$ abbreviates $\lambda_{i}(A)$ and, for each $i$,
$\hat{\lambda}_{i}:=\lambda_{i} - \lambda$.   


Continuing from \eqref{balladeno2}, we split the first term into two equal parts and recall the property of $G$ and $P$ given in \eqref{intpig}, thereby obtaining:
\begin{align*}I(u) - I(\ul)  \geq  \int_{\om} (\kappa/2) |\nabla u - \lambda \1|^q + & h'(\lambda^3)\hat{\l}_1 \hat{\l}_2 \hat{\l}_3 \,dx  +  \\ & \quad + \int_{\om} (\kappa/2) |\nabla u - \lambda \1|^q  +  \lambda h'(\lambda^3)G(\nabla u) \,dx \\
\geq  \int_{\om} (\kappa/2) |\Lambda(\nabla u) - \Lambda_0|^q + & h'(\lambda^3)\hat{\l}_1 \hat{\l}_2 \hat{\l}_3 \,dx  +  \\ & \quad + \int_{\om} (\kappa/2) |\nabla u - \lambda \1|^q  +  \lambda h'(\lambda^3)G(\nabla u) \,dx.
\end{align*}
Here, $\Lambda(A)$ is the $3-$vector with entries $\lambda_i(A)$ and $\lambda_{0}= (\lambda,\lambda,\lambda)$.  We have used the well-known inequality $|A-\lambda \1| \geq |\Lambda(A)-\Lambda_0|$.

In keeping with the notation introduced in \cite[Lemma 3.2]{BZ15}, let
\begin{align}\label{defF1}\scf_{1}(\Lambda)& =(\kappa/2)|\Lambda - \Lambda_0|^q + h'(\lambda^3)\hat{\lambda}_1 \hat{\lambda}_2 \hat{\lambda}_3,
\end{align}
and, in contrast to the approach of \cite{BZ15}, let
\begin{align*}\scf_{2}(A)& =(\kappa/2)|A-\lambda \1|^q + \lambda h'(\lambda^3)G(A).
\end{align*}

In these terms we then have
\begin{align}
\label{estf1f2}I(u)-I(\ul) & \geq \int_{\om} \scf_1 (\Lambda)\, dx + \int_{\om} \scf_2 (\nabla u) \,dx.
\end{align}
The sign of the first integral can be controlled by appealing to the following result:
\begin{lemma}(\cite[Lemma 3.3]{BZ15})\label{ff1}  The function $\scf_{1}(\Lambda)$ defined in \ref{defF1} is pointwise nonnegative on $\R^{+++}$ provided $h'(\l^3) > 0$ and 
\begin{align}\label{est1} \frac{(\kappa/2)}{h'(\lambda^3)\lambda^{3-q}} & \geq (q-2)^{(q-2)/2} q^{-q/2}.
\end{align}
\end{lemma}

The pointwise nonnegativity of $\scf_2$, on the other hand, relies primarily on the argument given in Lemma \ref{suffest2} below.  In short, the idea is that $\scf_2(\nabla u)$ is dominated by $|\nabla u-\lambda|^q$ for both small and large values of $|\nabla u-\lambda|$ provided $h'(\lambda^3)$ is itself not too large.  Thus we generate a new upper bound on $\lambda$ which must be imposed along with \eqref{est1} in order to guarantee that $I(u) \geq I(\ul)$.

\begin{lemma}\label{suffest2} With $G$ as defined in \eqref{defG} and for any positive constant $c_0$, let
\begin{align}\label{defm2}M_{2}(\lambda,c_0) & =\sup \left\{\frac{|G(A)|}{|A-\lambda \1|^2}: |A-\lambda \1| \geq c_0 \right\} \\
 \label{defm3}M_{3}(\lambda,c_0) &  = \sup\left\{ \frac{|G(A)|}{|A-\lambda \1|^3}:   0 < |A-\lambda \1| < c_0 \right\}. 
\end{align}
 Then $\scf_2(A) \geq 0$ for all $3 \times 3$ matrices $A$ provided 
\begin{align}\label{est2}\lambda h'(\lambda^3)  & \leq \min_{c_0}\left\{ (\kappa/2) \max\left\{ \frac{c_0^{q-2}}{M_{2}(\lambda,c_0)}, \frac{{c_0}^{q-3}}{M_{3}(\lambda,c_0)}\right\}\right\}.
\end{align} 
\end{lemma}
\begin{proof}By \eqref{partialpc2}, the quantity $M_3(\lambda,c_0)$ is finite and, in view of the at most quadratic growth of $G$, $M_3(\l,c_0)$ is uniformly bounded as a function of $c_0$.   $M_2(\l,c_0)$ has the same properties, but this time we appeal to the fact that $D_U G(\l \1)=0$ for all $U$.  

 Let $A \neq \l \1$ and let $c=|A-\lambda \1|$.  It is immediately clear that
\begin{align*}
\scf_2(A) \geq \frac{\kappa}{2}|A-\l \1|^q -\l h'(\l^3) |G(A)|,
\end{align*}
and we express the right-hand side in two ways:
\begin{align}\label{bbritten}
\frac{\kappa}{2}|A-\l \1|^q -\l h'(\l^3) |G(A)| &= \left(\frac{\kappa}{2}|A-\l \1|^{q-j} -\l h'(\l^3) \frac{|G(A)|}{|A-\l \1|^j}\right)|A-\l \1|^j,
\end{align}
where $j$ is either $2$ or $3$.   Now let $w:=2\l h'(\l^3)/\kappa$ and suppose that \eqref{est2} holds.  Then, in particular, 
\begin{align}\label{bakingparchment} w & \leq \max\{c^{q-2}/M_2(\l,c),  c^{q-3}/M_3(\l,c)\}.
\end{align}
If the maximum in \eqref{bakingparchment} is given by $c^{q-2}/M_2(\l,c)$ then $w M_2(\l,c) \leq c^{q-2}$, and hence $\frac{\kappa}{2} c^{q-2} - \l h'(\l^3) \frac{|G(A)|}{c^{q-2}} \geq 0$.  Using \eqref{bbritten} with $j=2$, we see that $\scf_2(A) \geq 0$.  If the maximum in \eqref{bakingparchment} is given by $c^{q-3}/M_3(\l,c)$ then we can argue similarly, this time using \eqref{bbritten} with $j=3$, to conclude that $\scf_2(A) \geq 0$.    
\end{proof}

\begin{lemma}\label{alassio} Let $f_1(c_0) = \frac{c_0^{q-2}}{M_2(\l,c_0)}$, $f_2(c_0) = \frac{c_0^{q-3}}{M_3(\l,c_0)}$
and define $\xi(c_0)=\max\{f_1(c_0),f_2(c_0)\}$.   Then $f_1$ is nondecreasing, $f_2$ is nonincreasing, and 
\begin{align*} \inf_{c_0 > 0}\xi(c_0) & = M_2(\l,c^*)^{q-3} M_3 (\l,c^*)^{2-q}
 \end{align*}
where $c^*$ is the unique fixed point of the function $c_0 \mapsto \frac{M_2(\l,c_0)}{M_3(\l,c_0)}$.
\end{lemma}

\begin{proof} Let $M_2(c_0)=M_2(\l,c_0)$ and $M_3(c_0)=M_3(\l,c_0)$ for brevity.  It is clear from their definitions that $M_2(c_0)$ and $M_3(c_0)$ are nonincreasing and nondecreasing respectively.   From this and the fact that $2 < q < 3$, it follows that $f_1$ is nondecreasing and $f_2$ is nonincreasing.   Since $f_1(c_0) \to +\infty$ as $c_0 \to +\infty$ and $f_2(c_0) \to +\infty$ as $c_0 \to 0+$, there is a unique point $c^*$ such that $f_1(c^*) = f_2(c^*)$,  
\begin{align*} \xi(c_0)=\left\{\begin{array}{ll} f_2(x_0) & \textrm{if} \ x_0 \leq c^* \\
f_1(c_0) & \textrm{if} \ x_0 \geq c^*\end{array}\right.\end{align*}
and where, moreover,  $\xi(c^*)$ is the global minimum of $\xi$ on $\R^+$.   It is straightforward to see that the condition $f_1(c^*)=f_2(c^*)$ is equivalent to the condition $c^*=M_2(c^*)/M_3(c^*)$, and that $f_1(c^*)=M_2(c^*)^{q-3} M_3(c^*)^{2-q}$.
\end{proof}

We are now in a position to state the main theorem of this section.

\begin{theorem}\label{printingpress} Let $W$ be as in \eqref{w} and suppose that $\lambda$ is chosen so that 
\begin{align}\label{estsuper} 0 \leq \lambda^{3-q} h'(\lambda^3) & \leq  \min\{ (\kappa/2) (q-2)^{(2-q)/2} q^{q/2}, (\kappa/2)\lambda^{2-q}M_2(\l,c^*)^{q-3} M_3 (\l,c^*)^{2-q}\},
\end{align}
where $c^*$ is the unique fixed point of the function $c_0 \mapsto M_2(\l,c_0)/M_3(\l,c_0)$.   Then $I(u) \geq I(\ul)$ for any map $u \in H^{1}(\om,\R^3)$ whose boundary values agree with those of $\ul$ in the sense of trace.  In particular, the largest possible value $\lambda_0$ satisfying \eqref{estsuper} is a lower bound for $\lambda_{\textrm{crit}}$.  Moreover, if \eqref{estsuper} holds with strict inequality then there is $C=C(\om) >0$ such that 
\begin{align}\label{ocular}  I(u) - I(\ul) \geq C \int_{\om} |\nabla u - \l 1|^q  + |G(\nabla u)|\,dx.\end{align}
for all admissible $u$.  
\end{theorem}
\begin{proof} Acccording to \eqref{estf1f2}, $I(u)-I(\ul)$ is bounded below by the sum of $\int_{\om} \scf_1(\nabla u) \, dx$ and $\int_{\om} \scf_2(\nabla u)\,dx$.  By inequality \eqref{estsuper} and Lemma \ref{ff1}, the first of these integrals is nonnegative, while Lemmas \ref{suffest2} and \ref{alassio} together imply that the second integral is nonnegative.   Either way, it follows that $I(u) \geq I(\ul)$, as claimed.  It is then clear that $\lambda_0$, as defined above, is not larger than $\l_{\textrm{crit}}$.

Now suppose that \eqref{estsuper} holds with strict inequality.   The proof of \cite[Theorem 3.6]{BZ15} shows that the inequality $\l^{3-q}h'(\l^3) < \frac{\kappa}{2} (q-2)^{\frac{2-q}{2}} q^{\frac{q}{2}}$ implies, for some $\eps > 0$, that  
\begin{align*} \scf_1(A) & \geq \eps |\Lambda(A) - (\l,\l,\l)|^q \end{align*}
for all $A$ in $\R^{3 \times 3}$.   Here, $\Lambda(A)=(\l_1(A),\l_2(A),\l_3(A))$ is the vector of singular values of the matrix $A$.   The rigidity argument, with minor modifications, given in \cite[Theorem 3.6]{BZ15} then shows that there is a constant, $\beta(\om)$, say, such that
\begin{align*} \int_{\om}  |\Lambda(\nabla u) - (\l,\l,\l)|^q \,dx & \geq \beta(\om) \int_{\om} |\nabla u - \l \1|^q \,dx. \end{align*}   Hence
\begin{align*} 
\int_{\om} \scf_1(\nabla u) \,dx & \geq \eps \beta(\om) \int_{\om} |\nabla u - \l \1|^q \,dx.
\end{align*}
Finally, we deal with the term involving $\scf_2(A)$.  Fix $A$ in $\R^{3 \times 3}$ and let $c_0 = |A-\l \1|$.  According to the proof of Lemma \ref{suffest2}, 
\begin{align}\label{telemannfantasia}\scf_2(A) \geq \frac{\kappa}{2}\left(|A-\l \1|^{q-j} -\frac{2\l h'(\l^3)}{\kappa} \frac{|G(A)|}{|A-\l \1|^j}\right)|A-\l \1|^j
\end{align}
for $j=2$ and $3$.   Reusing the notation $w=2\l h'(\l^3)/\kappa$, and applying the strict version of \eqref{estsuper}, there is $\eps'>0$, which is independent of $c_0$, such that 
\begin{align}\label{tafelmusik} w+\eps' \leq \max\{c_0^{q-2}/M_2(\l,c_0),  c_0^{q-3}/M_3(\l,c_0)\}.
\end{align} 
By rewriting \eqref{telemannfantasia}, we obtain 
\begin{align*}\scf_2(A) \geq \frac{\kappa}{2}\left(c_0^{q-j} -(w+\eps') M_j(\l,c_0)\right)c_0^j + \frac{\kappa\eps'}{2}|G(A)|
\end{align*}
for $j=2$ and $3$.   Thanks to \eqref{tafelmusik}, the term in brackets is nonnegative, which leaves $\scf_2(A) \geq \kappa\eps'|G(A)|/2$.  Both terms in the right-hand side of inequality \eqref{ocular} are now accounted for.
\end{proof}

\begin{remark}\emph{The goal of Theorem \ref{printingpress} is to give the largest possible bound on $\l$ such that $I(u) \geq I(\ul)$.    A careful look at the proof of Lemma \ref{suffest2} shows that one could replace $|G(A)|$ by $G^{-}(A)=-\min\{G(A),0\}$ and that the same conclusions would result, but with 
\begin{align*} M^{-}_2(\l,c_0) & = \sup\left\{\frac{G^{-}(A)}{|A-\l\1|^2}: |A-\l\1|\geq c_0\right\}
\\
 M^{-}_{3}(\lambda,c_0) &  = \sup\left\{ \frac{G^{-}(A)}{|A-\lambda \1|^3}:   0 < |A-\lambda \1| < c_0 \right\}
\end{align*}
in place of $M_2(\l,c_0)$ and $M_3(\l,c_0)$ respectively.  Since $M_{j}^{-}(\l,c_0) \leq M_j(\l,c_0)$ for $j=2$ and $3$, it follows that the upper bound involving $c^*$ in \eqref{estsuper} would not decrease.   Numerical evidence suggests that it does, in fact, increase, and thus provides a better (lower) bound for $\l_{\textrm{crit}}$.  See Remark \ref{silvermaple} and Section \ref{a3} for further details.}
\end{remark}

\subsection{Calculations leading to a concrete upper bound}\label{concrete}

The upper bound \eqref{estsuper} given in the statement of Theorem \ref{printingpress} contains two terms, one of which is explicitly given in terms of the exponent $q$ and one which depends on the fixed point $c^*$ of the function $c_0 \mapsto M_2(\l,c_0)/M_3(\l,c_0)$.   While it does not seem to be possible to find $c^*$ purely analytically, one can nevertheless make progress using a mixture of analysis and a careful numerical calculation, as we now describe.

\begin{proposition}\label{coleridge}  With $M_2(\l,c_0)$ as defined by \eqref{defm2},    $\lim_{c_0 \to \infty} M_2(\l,c_0) = \sqrt{2}$.
\end{proposition}
\begin{proof}  First, note that the limit in the statement exists because $M_2(c)$ is nonincreasing and bounded below.   Now 
\begin{align*}\frac{G(A)}{|A-\l\1|^2} & = K(\hat{A}) +\frac{\l(2 \tr A - 3 \l)}{|A-\l\1|^2}K(\hat{A})-\frac{\l}{|A - \l\1|^2} \left(\sum_{i=1}^{3} \l_i(A) - \tr A \right)
\end{align*}
where $\hat{A} = A/|A|$ and 
\begin{align*} K(\hat{A}) = \sum_{1 \leq i < j \leq 3} \l_i(\hat{A})\l_j(\hat{A}) - \tr \cof \hat{A}.
\end{align*}
Let $c_0 > 0$.  Since $K$ is bounded and the term $\sum_{i=1}^{3} \l_i(A) - \tr A$ has at most linear growth in $A$, it is clear that there are constants $\alpha_1$ and $\alpha_2$, depending only on $\l$, such that
\begin{align}\label{callas}
|K(\hat{A})| + \frac{\alpha_1}{c_0} \leq \frac{|G(A)|}{|A-\l\1|^2} & \leq  |K(\hat{A})| + \frac{\alpha_2}{c_0}
\end{align}
whenever $|A-\l \1| \geq c_0$.  We now prove that (a) $\max\{K(\hat{A})\}=\sqrt{2}$, where the maximum is taken over all unit matrices $\hat{A}$, and (b) that given any $c>0$ there is $A(c)$ such that $|A(c)-\l \1| \geq c$ and $K(\widehat{A(c)})=\sqrt{2}$.  The proposition then follows from this and \eqref{callas}.

To simplify the notation we replace $\hat{A}$ by $A$ in the following.   By the polar factorization theorem (see \cite[Theorem 3.2-2]{Ci04}), there is a positive semidefinite and symmetric matrix $U$ and a matrix $R \in O(3)$ such that $A = R U$.  The eigenvalues of $U$ are the singular values $\l_i(A)$, $i=1,2,3$, and one has
\begin{align*}
K(A) & = \sum_{1 \leq i,j \leq 3} \l_i(A) \l_j(A) -\tr (\cof R \, \cof U) \\
& = \sum_{1 \leq i,j \leq 3} \l_i(A) \l_j(A) -\tr( (\det R) \, R \, Q^T \, \cof D \, Q),
\end{align*}
where the well-known decomposition $U=Q^T D Q$ has been used, with $D = \textrm{diag} (\l_1(A),\l_2(A),\l_3(A))$ and $Q$ in $O(3)$. Note that $\det R = \pm 1$.     It follows that 
\begin{align*} K(A) & = \sum_{1 \leq i,j \leq 3} \l_i(A) \l_j(A) -\tr(P \, \textrm{diag} (\l_2\l_3,\l_1\l_3, \l_1 \l_2)), \end{align*}
where $P:=Q \det R \, R \, Q^T$ belongs to $SO(3)$.    The term involving the trace satisfies 
\begin{align*}
\tr(P \, \textrm{diag} (\l_2\l_3,\l_1\l_3, \l_1 \l_2) = P_{11}\l_2\l_3 +P_{22} \l_1\l_3 + P_{33} \l_1 \l_2.
\end{align*}
The term $\sum_{1 \leq i,j \leq 3} \l_i(A) \l_j(A)$ is independent of $P$, so we can vary $P$ in order to minimize the term involving the trace given above.    Since $P$ belongs to $SO(3)$, its rows and columns are orthogonal unit vectors.  In particular, $|P_{11}| \leq 1$, so to make $P_{11} \l_2 \l_3$ minimal we should take $P_{11} = -1$. (Here we have used the ordering $\l_3 \geq \l_2 \geq \l_1$.)   Hence row $1$ of $P$ is $(-1,0,0)$, forcing the second and third rows to take the form 
$(0,P_{22},P_{33})$ and $(0,P_{32},P_{33})$ respectively.   In particular, $\det P = -(P_{22}P_{33}-P_{23}P_{32}) = 1$.  Without loss of generality, we may take $P_{22}=\cos \alpha$, $P_{32}=P_{23}=\sin \alpha$, $P_{33}=-\cos \alpha$ for some $\alpha$, in which case
\begin{align*}
P_{11}\l_2\l_3 +P_{22} \l_1\l_3 + P_{33} \l_1 \l_2 & \geq -\l_2 \l_3 + \cos \alpha (\l_1 \l_3 - \l_1\l_2) \\
& \geq -\l_2 \l_3 - \l_1 (\l_3 - \l_2)
\end{align*} 
by choosing $\alpha = \pi$.  Hence
\begin{align*}  K(A) \leq 2 \l_3 (\l_1 + \l_2)
\end{align*}
where the matrix $P$ yielding this upper bound is given by $P=\textrm{diag}(-1,-1,1)$.   Bearing in mind that $|A|^2 = \sum_{i=1}^{3} \l_i^2 = 1$, it can be shown that $2 \l_3 (\l_1 + \l_2) \leq \sqrt{2}$, and that a maximizing choice of singular values is $\l_1=\l_2=1/2$, $\l_3 = 1/\sqrt{2}$.  A suitable choice for a maximizing $A$ would therefore be $A_0 = P \,\textrm{diag} (1/2,1/2,1/\sqrt{2}) = \textrm{diag}\, (-1/2,-1/2,1/\sqrt{2})$, and one can check that $K(A_0)=\sqrt{2}$.  To conclude the proof, it is enough to choose $A(c):=r A_0$ for any $r=r(c)$ large enough that $|A(c) -\l\1| \geq c$.
\end{proof}

Since $D^2G(\l \1)=0$, it can easily be shown that $M_2(\l,c_0)$ is constant as a function of $c_0$ for all sufficiently small and positive $c_0$.   We are therefore justified in writing $M_2(\l,0):=\lim_{c_0 \to 0}M_2(c_0)$, and in fact 
\begin{align*}M_2(\l,0) = \sup\left\{\frac{|G(A)|}{|A-\l\1|^2}: \ A \neq \l \1\right\}.
\end{align*}
According to Proposition \ref{coleridge}, we must have 
\begin{align}\label{frankbridge} M_2(\l,0) & \geq M_2(\l,c_0) \geq \sqrt{2} \ \textrm{for all} \ c_0 > 0.\end{align}  Using a `brute force' approach, which we describe in the appendix, we find that $\sqrt{2} - M_2(\l,0) \approx 2.7\times 10^{-9}$, independently of $\l$ in the range $[1,2]$.   Thus, in view of \eqref{frankbridge}, it seems that $M_2(\l,c_0)\equiv \sqrt{2}$ for all $c_0$ and $\l$, and we record this as:
\begin{conjecture}\label{conjm2}\emph{For all $\l >0$ and $c_0>0$, $M_2(\l,c_0)=\sqrt{2}$.}
\end{conjecture}

In order to explain the method used to approximate $M_2(\l,0)$, we introduce the notation 
\begin{align*}m_l(A, \lambda) := \frac{|G(A)|}{|A - \lambda{\mathbf{1}}|^l}\end{align*}
for $l=2$ and $3$, and we restrict attention to $\l$ in the range $[1,2]$.   Choose a random starting matrix $A_1$,  compute $g_1:=\nabla m_2(A_1,\l)$ and find the scalar $\sigma_1$, say, which maximises $m_2(A_1+\sigma g_1)$.  Let the maximising value of $\sigma$ be $\sigma_1$;  then we set $A_2 = A_1 + \sigma_1 g_1$. Proceeding iteratively, we
compute a sequence of matrices $A_i$. It turns out that $m_2(A_i, \lambda)$
tends to $\sqrt{2}$ as $i$ increases, regardless of $\lambda$.

Supposing that Conjecture \ref{conjm2} is right, we are now required to find $M_2(\l,c^*)^{q-3} M_3 (\l,c^*)^{q-2}=2^{\frac{q-3}{2}}M_3(\l,c^*)$, where $c^*$ satisfies
$c^* = \sqrt{2}/M_3(\l,c^*)$.   We begin by recalling that $c_0 \mapsto M_3(\l,c_0)$ is nondecreasing, and that, thanks to the at most quadratic growth of the function $G(A)$, there is $c_1(\l)$ such that $M_3(\l,c_0)=M_{3}(\l,c_1)$ for all $c_0 \geq c_1$.   We are therefore justified in defining $M_3(\l,\infty):=M_3(\l,c_1)$.   
Appealing again to numerical techniques, which we describe in the appendix, we find that, for $\l$ in the range $[1,2]$, there is very good agreement between $c_1(\l)$ and the expression $\nu_3\l + \nu_2$, where $\nu_2 \approx 1.764\times 10^{-3}$ and 
and $\nu_3 \approx 1.842$. Moreover, for the same range of $\l$, there is strong numerical evidence for the approximation $M_3(\l,\infty) \approx \nu_1/\l$, where $\nu_1 \approx 0.4501$. See Fig 1. 

\begin{figure}[h]
\begin{center} 
\includegraphics*[width=3.8in]{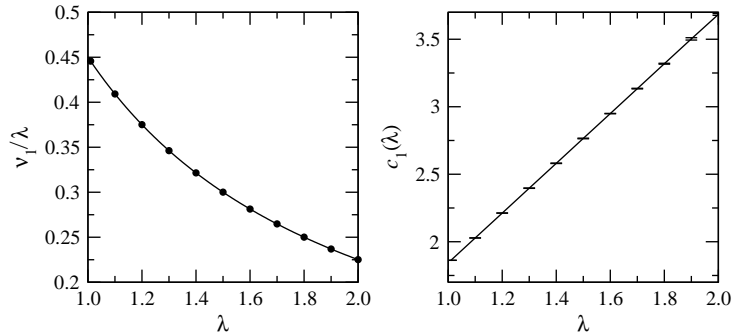}
\caption{Left: plot of $\nu_1/\l=0.4501/\l$ (continuous curve) fitted to a series of values of $M_3(\l,\infty)$ (filled circles). Right: plot of $c_1(\lambda)$ versus $\lambda$. The error bars, which are almost too narrow to see, show the
maximum and minimum values of $c_1$ from among the fifteen top-ranking
matrices. (See the appendix for details.)}
\end{center}
\label{czero}
\end{figure} 

This leads naturally to the following:

\begin{conjecture}\label{conjm3} $M_3(\l,c_0) = M_3(\l,c_1)=\nu_1/\l$ for all $c_0\geq c_1(\l)$ and for $\l$ in the range $[1,2]$, where $c_1(\l)=\nu_3\l + \nu_2$ and the values of $\nu_1,\nu_2$ and $\nu_3$ are given above. \end{conjecture}

Let us now suppose that Conjectures \ref{conjm2} and \ref{conjm3} are correct.   Note that then the function $p: c \mapsto M_2(\l,c)/M_3(\l,c)=\sqrt{2}\l/\nu_1$ is independent of $c$ for all $c \geq c_1$.   There are thus two possibilties for the fixed point $c^*$ of $p$:  either $c^*=\sqrt{2}\l/\nu_1$ or $c^* < c_1$.   Suppose for a contradiction that $c^*< c_1$.   Then $M_3(\l,c^*)\leq M_3(\l,c_1)$, and so $p(c^*) \geq p(c_1)$.  But $p(c^*)=c^*$ and $p(c_1)=\sqrt{2}\l/\nu_1$, which gives $c^* \geq\sqrt{2}\l/\nu_1$.  By hypothesis, $c_1>c^*$, which when combined with the preceding inequality implies $c_1 > \sqrt{2}\l/\nu_1$.  Applying Conjecture \ref{conjm3} and rearranging, we see that this is equivalent to $\nu_1\nu_2 \geq (\sqrt{2}-\nu_3\nu_1)\l$, which can only hold if $\l \leq 1.357 \times 10^{-3}$.   But we supposed that $\l \geq 1$, which is a contradiction.    In summary, we have shown the following result.

\begin{proposition}\label{identifyc*} Let $\l$ belong to the interval $[1,2]$ and suppose that Conjectures \ref{conjm2} and \ref{conjm3} are correct.  Then the unique fixed point of the function $p(c):=M_2(\l,c)/M_3(\l,c)$ is given by $c^* =\sqrt{2}\l/\nu_1$, and 
\begin{align*} M_2(\l,c^*)^{q-3} M_3(\l,c^*)^{2-q}=\frac{1}{\sqrt{2}}\left(\frac{\sqrt{2}\l}{\nu_1}\right)^{q-2}.
\end{align*}
\end{proposition}

Referring back to the upper bound given in \eqref{estsuper}, we now have the following:

\begin{corollary}\label{concretebound}  Let the assumptions of Proposition \ref{identifyc*} hold.  Then $I(u) \geq I(\ul)$ provided
\begin{align}\label{reduced} 0 & \leq \lambda^{3-q} h'(\lambda^3)  \leq  (\kappa/2)(\sqrt{2})^{q-3}{\nu_1}^{2-q}.\end{align}
\end{corollary}

\begin{proof}  It is enough to show that
\begin{align*} \min\{(q-2)^{(2-q)/2} q^{q/2},\lambda^{2-q}M_2(\l,c^*)^{q-3} M_3 (\l,c^*)^{2-q}\} & = (\sqrt{2})^{q-3}{\nu_1}^{2-q}.\end{align*}
By Proposition \ref{identifyc*}, we clearly have $\lambda^{2-q}M_2(\l,c^*)^{q-3} M_3 (\l,c^*)^{2-q} =  (\sqrt{2})^{q-3}{\nu_1}^{2-q}$.  Therefore it remains to show that  $(q-2)^{(2-q)/2} q^{q/2} > (\sqrt{2})^{q-3}{\nu_1}^{2-q}$ for $2 < q <3$.  Let $z(q)=(\sqrt{2})^{q-3}{\nu_1}^{2-q}$ and $y(q)=
(q-2)^{(2-q)/2} q^{q/2}$.  Note that $z$ is convex, while a short calculation reveals that $y$ is concave on the interval $(2,3)$.  Therefore the inequality $z(q) < y(q)$ will follow from the pair of inequalities $z(2)<y(2+)$ and $z(3)<y(3)$, both of which are easy to check.  This concludes the proof.  
\end{proof}

\begin{remark}\label{silvermaple}\emph{Corollary \ref{concretebound} suggests that the bound is not the best possible:  we would really expect both terms in the bound given in \eqref{estsuper} to play a role.  One way to achieve this might be to use the quantities $M_j^{-}$ in place of $M_j$ for $j=2,3$.   Indeed, we find, again numerically, that there is very good agreement between $M_{3}^{-}(\l,\infty)$ and $\nu_1'/\l$ for $\l$ in the range $[1,2]$, where $\nu_1' = 0.1923$.  Interestingly, if we then replace $M_3$ by $M_3^{-}$ everywhere in the preceding calculations, we find that $z(3) < y(3)$ no longer holds.   In other words, \emph{both} terms in the upper bound given by \eqref{estsuper} appear to contribute.  This observation comes with some caveats, however; see Section \ref{a3} in the appendix.}
\end{remark}

\section{A condition for the inequality $I(u) \leq I(\ul)$}\label{spectral}

The results in the previous sections provide conditions on $\l$ under which the inequality $I(u) \geq I(\ul)$ holds for admissible maps.  It is natural to ask what information results from supposing that $I(u)\leq I(\ul)$, and, of our results, Theorem \ref{printingpress} is the first place to look.  Now, if \eqref{estsuper} holds with strict inequality, then $\ul$ sits in a `potential well', as expressed by the estimate \eqref{ocular}, which we recall here for the reader's convenience:
\begin{align*} I(u) - I(\ul) \geq C \int_{\om} |\nabla u - \l 1|^q  + |G(\nabla u)|\,dx.\end{align*}
Thus, in these circumstances, $I(u) \leq I(\ul)$ is impossible.  If \eqref{estsuper} holds with equality then a similar remark applies, but with the additional possibility of losing one or both terms in the right-hand side of \eqref{ocular}.  And when \eqref{estsuper} fails, the preceding analysis tells us nothing about those $u$ whose energy satisfies $I(u) \leq I(\ul)$.  Therefore a different approach is called for. 

Consider the following simplified model, in which we set the function $Z$ appearing in \eqref{w} to zero.   Thus we let
\begin{align}\label{ww} W(A) & = |A|^{q} + h(\det A) \ \ \textrm{for} \ A \in \R^{3 \times 3},
\end{align}
where, as before, $2 < q < 3$.    Let $\l > 1 $ and note that, by the convexity of $h$, 
\begin{align}\nonumber I(u) - I(\ul) & \geq \int_{\om} h'(\l^3)(\det \nabla u - \l^3) + |\nabla u|^q - |\nabla \ul|^q \,dx \\
\label{goatwillow}
& \geq \int_{\om} h'(\l^3)(\det \nabla u - \l^3) + \kappa|\nabla u-\nabla \ul|^q\,dx,\end{align}
where we have used \eqref{c:q} and \eqref{ineq0:3d}, and where the constant $\kappa$ obeys the bounds specified in inequality \eqref{j:kappalimits}.
Using identity \eqref{greenalder} and the definition \eqref{f1} of $P$, write 
\begin{align}\label{creepingwillow} \det \nabla u -\l^3 & = \hat{\l}_1\hat{\l}_2\hat{\l}_3 + \l P(\nabla u),\end{align}
where $\hat{\l}_i:=\l_i-\l$ and $\l_i:=\l_i(\nabla u)$.   Finally, recall that the function $G$ defined in \eqref{defG} satisfies
\begin{align*}
\int_{\om} G(\nabla u) \,dx &  = \int_{\om} P(\nabla u) \,dx
\end{align*}
whenever $u$ is admissible.   Combining this with \eqref{goatwillow} and \eqref{creepingwillow}, we have
\begin{align}\label{almondwillow} I(u) - I(\ul) & \geq \int_{\om} h'(\l^3)(\hat{\l}_1\hat{\l}_2\hat{\l}_3 + \l G(\nabla u)) + \kappa|\nabla u-\nabla \ul|^q\,dx.\end{align}
It will be useful to have a shorthand for the function with prefactor $h'(\l^3)$ appearing in \eqref{almondwillow};  accordingly, let
\begin{align}\label{defH} H(A) & = \hat{\l}_1(A)\hat{\l}_2(A)\hat{\l}_3(A) + \l G(A). \end{align}

We now give a series of results which allow us to find a lower bound on the function $H$.  

\begin{lemma} Let $A$ be a $3 \times 3$ matrix such that $\det A > 0$ and whose singular values obey $\l_1(A) \leq \l_2(A) \leq \l_3(A)$.  Then there is $R$ in $SO(3)$ such that 
\begin{align} \label{abcform}G(A) & =(1-R_{11})\alpha + (1-R_{22})\beta + (1-R_{33})\gamma,  \end{align} 
where $\alpha = \l_2 \l_3-\l\l_1$, $\beta = \l_1 \l_3 - \l \l_2$ and $\gamma = \l_1 \l_2 - \l \l_3$, and $\l_j=\l_j(A)$ for $j=1,2,3$.

\end{lemma}
\begin{proof}  For brevity, write $\l_j$ in place of $\l_j(A)$ for $j=1,2,3$.  Note that $P(A) = \alpha + \beta + \gamma$ by definition, and that, by polar decomposition, there are matrices $Q_1$ and $Q_2$, belonging to $O(3)$, such that $A= Q_2 Q_1^T D Q_1$ and where $D$ is a diagonal matrix with entries $\l_1$, $\l_2$ and $\l_3$.   (See \cite[Theorem 3.2-2]{Ci04}.)  Since $\det A > 0$, $Q_2$ must belong to $SO(3)$.   We see that
\begin{align*} N(A) & = \tr \cof (Q_2 Q_1^T D Q_1) - \l \tr Q_2 Q_1^T D Q_1 \\
& = \tr (R (\cof D - \l D)),\end{align*}
where $R:=Q_1 Q_2 Q_1^T$ belongs to $SO(3)$.   Hence $N(A) = R_{11} \alpha + R_{22} \beta + R_{33} \gamma$, and \eqref{abcform} follows.  
\end{proof}

Our aim is to minimize $G(A)$ by allowing $R$ to vary in $SO(3)$.  To that end, consider the following.  

\begin{lemma}\label{lucombeoak} Let $R$ belong to $SO(3)$ and suppose that it minimizes \begin{align*}g(R)&= (1-R_{11})\alpha + (1-R_{22})\beta + (1-R_{33})\gamma.\end{align*}  Then 
\begin{align}\label{baywillow1} \alpha R_{12} -\beta R_{21} & = 0 \\
\label{baywillow2}\alpha R_{13} -\gamma R_{31} & = 0 \\
\label{baywillow3} \beta R_{23} - \gamma R_{32} & = 0. \end{align}
If none of $\alpha,\beta,\gamma$ is zero then 
\begin{align}\label{crack-willow1} \left(\left(\frac{\alpha}{\beta}\right)^2-1\right)R_{12}^2 + \left(\left(\frac{\alpha}{\gamma}\right)^2-1\right)R_{13}^2 & = 0 \\
\label{crack-willow2} 
\left(\left(\frac{\beta}{\gamma}\right)^2-1\right)R_{23}^2 - \left(\left(\frac{\alpha}{\beta}\right)^2-1\right)R_{12}^2 & = 0. 
\end{align}
Moreover, if exactly one of $\alpha, \beta, \gamma$ is zero, then  either $R_{11}^2=1$ or $R_{33}^2=1$.   The same is true if none of $\alpha, \beta$ and $\gamma$ is zero, provided at least one of $\alpha > \beta$ and $\beta > \gamma$ holds. 
\end{lemma}
\begin{proof}  We first show that \eqref{baywillow1}-\eqref{baywillow3} hold.  It is well known that the tangent space to $SO(3)$ at $R$ consists of those $\rho$ in $R^{3 \times 3}$ such that $R^T \rho$ is antisymmetric.   From this, it easily follows that there are real numbers $a,b$ and $c$ such that 
\begin{align*} \rho_{11} & = -b R_{13} - a R_{12} \\
\rho_{22} & = a R_{21} - c R_{23}\\
\rho_{33} & = b R_{31}  + c R_{32}.
\end{align*}
Now suppose that $R(\eps)$ is a smooth path of matrices belonging to $SO(3)$, and satisfying $R(0)=R$ and $\dot{R}(0)=\rho$.   We then have
\begin{align*} \partial_{\eps}\arrowvert_{\eps = 0} g(R(\eps))&= a(R_{12} \alpha - R_{21}\beta) + b (R_{13}\alpha - R_{31}\gamma) + c(R_{23}\beta -R_{32} \gamma),\end{align*}
so that, by varying $a,b$ and $c$ independently, the stationarity conditions \eqref{baywillow1}-\eqref{baywillow3} follow.  

To prove the last part of the statement, we consider cases as follows.   

\vspace{1mm}
\noindent Case (i): $\bm{\alpha=0 > \beta \geq \gamma}.$
\newline \noindent Using \eqref{baywillow1} and \eqref{baywillow2}, we see that $R_{21}=R_{31}=0$.  Hence, since the first column of $R$ is a unit vector, we must have $R_{11}^2=1$. 

\vspace{1mm}
\noindent Case (ii):  $\bm{\alpha> 0=\beta > \gamma}.$
\newline \noindent Using \eqref{baywillow1} and \eqref{baywillow3}, we must have $R_{12}=0$ and $R_{31}=0$.  By \eqref{baywillow2}, we then have $R_{13}=0$.  Hence $R_{11}^2=1$, as before.

\vspace{1mm}
\noindent Case (iii):  $\bm{\alpha \geq \beta > \gamma=0.}$  
\newline \noindent Equations \eqref{baywillow2} and \eqref{baywillow3} imply respectively that $R_{13}=0$ and $R_{23}=0$.  Hence $R_{33}^2=1$.

\vspace{1mm}
\noindent Case (iv):  $\bm{\alpha >  \beta > \gamma; \alpha \neq 0, \beta \neq 0, \gamma \neq 0}.$   In this case, \eqref{baywillow1} implies that $R_{12} = R_{21} = 0$, and so $R_{11}^2=1$.

\vspace{1mm}
\noindent Case (v):  $\bm{\alpha = \beta > \gamma; \alpha \neq 0, \beta \neq 0, \gamma \neq 0}.$  Now \eqref{crack-willow1} implies that $R_{13}=0$ and \eqref{crack-willow2} that $R_{23}=0$.  Hence $R_{33}^2=1$.

\vspace{1mm}
\noindent Case (vi):  $\bm{\alpha >  \beta = \gamma; \alpha \neq 0, \beta \neq 0, \gamma \neq 0}.$   Equations \eqref{crack-willow1} and \eqref{crack-willow2} imply that $R_{12}=R_{13}=0$, and hence $R_{11}^2=1$.
\end{proof}

The next result will enable us to deal with the case $\alpha = \beta = \gamma$. 

\begin{lemma}\label{mintrace} Let $R$ belong to $SO(3)$ and suppose that $R$ is symmetric.  Then $3 \geq \tr R \geq -1$.
\end{lemma}
\begin{proof}If $R$ is a symmetric, orthogonal matrix then $R^2=\1$, from which it follows that any eigenvalue $\mu_i$ of $R$ must satisfy $\mu_i^2=1$.  Moreover, $\mu_1 \mu_2 \mu_3=1$, from which it follows that $3 \geq \tr R = \mu_1+\mu_2+\mu_3 \geq -1$. 
\end{proof}

\begin{lemma}\label{cricket-batwillow}  Let $R$ belong to $SO(3)$.
\begin{align*}g(R)= (1-R_{11})\alpha + (1-R_{22})\beta + (1-R_33)\gamma
\end{align*}
satisfies 
\begin{align}\label{italianalder} g(R) &\geq \min\{2(\beta + \gamma),0\}.\end{align}
\end{lemma}
\begin{proof}  If two or more of $\alpha, \beta$ and $\gamma$ are zero then the lower bound $g(R) \geq 2(\beta + \gamma)$ is trivial.  Thus, to minimize $g$, we may begin by supposing that the conditions of Lemma \ref{lucombeoak} apply, so that, in all cases except $\alpha = \beta = \gamma$, we have either that $R_{11}^2=1$ or $R_{33}^2=1$.     First suppose that $R_{11}^2 = 1$.  Then the diagonal elements of $R$ are either of the form $1,\cos \sigma, \cos \sigma$ for some $\sigma$, or else of the form $-1, \cos \sigma, - \cos \sigma$.  In the former case, 
\begin{align*} g(R)  = (1-\cos \sigma)(\beta + \gamma).\end{align*}
If $\beta+ \gamma  \geq 0$ then clearly $g(R) \geq 0$. If $\beta + \gamma <0$ then to minimize $g$ we take $\cos \sigma = -1$ and the claimed lower bound follows.  If $R_{11}=-1$, then 
\begin{align*} g(R) = 2 \alpha + \beta + \gamma + (\gamma - \beta) \cos \sigma, \end{align*}
which, since $\gamma \leq \beta$, implies that we should take $\cos \sigma = 1$ in order to minimize $g$.  Hence 
\begin{align*} g(R) & \geq 2 (\alpha + \gamma) \geq 2(\beta + \gamma).\end{align*}
If $R_{33}^2=1$ then the argument needed is similar.  Finally, let us suppose that $\alpha = \beta = \gamma$.  Then 
\begin{align*}g(R) = (3 -\tr R) \gamma.\end{align*}
If $\gamma < 0$ then $g$ is minimized when $\tr R = -1$, according to Lemma \ref{mintrace}.  Hence, in this case, $g(R) \geq 4 \gamma = 2(\beta + \gamma)$.  Otherwise, $g(R) \geq 0$ because $\tr R \leq 3$ by Lemma \ref{mintrace}.     This completes the proof.
\end{proof}

\begin{proposition}\label{bartok1}  Let $\l > 0$, let $H$ be given by \eqref{defH} and let $\l_1 \leq \l_2 \leq \l_3$ be the singular values of $A$.   Then 
\begin{align*}H(A) & \geq \left\{ \begin{array}{l l} (\l_1-\l)(\l_2-\l)(\l_3-\l)  & \textrm{if } \l_1 \geq \l   \\    
(\l_1-\l)(\l_2+\l)(\l_3+\l) & \textrm{if } \l_1 \leq \l. \end{array}\right.
\end{align*}
\end{proposition}
\begin{proof} First suppose that $\l_1 \geq \l$.   Note that $\beta+\gamma= (\l_1 - \l)(\l_2+\l_3)$ is then nonnegative, and so, by inequality \eqref{italianalder} in Lemma \ref{cricket-batwillow},  we have $G(A)=g(R) \geq 0$.  Hence $H(A)=(\l_1-\l)(\l_2-\l)(\l_3-\l) + \l G(A) \geq (\l_1-\l)(\l_2-\l)(\l_3-\l)$.   

Now suppose that $\l_1 \leq \l$.  Then the lower bound in \eqref{italianalder} is $2(\beta+\gamma)\leq 0$, and so 
\begin{align*} H(A) & \geq (\l_1-\l)(\l_2-\l)(\l_3-\l) +2\l (\l_1-\l)(\l_2+\l_3) \\
& = (\l_1-\l)(\l_2+\l)(\l_3+\l).
\end{align*}
\end{proof}

The main result of this subsection is the following.  
\begin{proposition}\label{lime} Suppose $h'(\l^3) \geq 0$ and let $W$ be given by \eqref{ww}.  Then any admissible map $u \neq \ul$ is such that: 
\begin{align*} I(u) -I(\ul) & \geq \int_{\om} \kappa|\nabla u-\nabla \ul|^q  + h'(\l^3) H(\nabla u) \,dx. 
\end{align*}
In particular, if $\int_{\om}H(\nabla u) \,dx \geq 0$ then $I(u) > I(\ul)$, while if $I(u) \leq I(\ul)$ then 
\begin{align} \nonumber
\int_{\{x \in \om:  \ \l_1(\nabla u (x)) \geq \l\}}h'(\l^3)(\l_1-\l)(\l_2-\l)(\l_3 - \l) + & \kappa|\nabla u-\nabla \ul|^q \,dx  \\
& \label{cornishelm}\leq \int_{\{x \in \om:  \ \l_1(\nabla u (x)) \leq \l\}} h'(\l^3) (\l - \l_1)(\l+\l_2)(\l+\l_3)\,dx.
\end{align}
\end{proposition}
\begin{proof}  This follows from \eqref{almondwillow} and Proposition \ref{bartok1}. 
\end{proof}

We remark that the results of Section \ref{plotselm} imply that inequality \eqref{cornishelm} ought \emph{not} to be possible for $\l$ such that $\l^{3-q}h'(\l^3)$ is sufficiently small (see \eqref{estsuper}).  It is not immediately obvious from \eqref{cornishelm} why this should be so;  nor is it clear why such a prominent role is played by the smallest singular value $\l_1(\nabla u)$.  This surely warrants further investigation.

\appendix
\section*{Appendix}
\renewcommand{\thesection}{A} 
Two different algorithms have been used to compute the quantity $M_3(\l,\infty):=\lim_{c \to \infty}M_3(\l,c)$, as defined in Subsection \ref{concrete}.  We recall the notation 
\begin{align*}m_{l}(\l,A)=\frac{|G(A)|}{|A-\l\1|^{l}}\end{align*}
for $l=2,3$.  The algorithms are also brought to bear on the problem of calculating $M_3^{-}(\l,\infty)$, and we summarise the results below.

\subsection{Algorithm A: conjugate gradient}

This is a `brute force' approach, which consists of 
\begin{itemize}
\item Choosing a value of $\lambda\in[1, 2]$;
\item Generating matrices $A$ with all elements $a_{ij}$ being uniformly
distributed random numbers over the interval $[-\alpha, \alpha]$,
for given $\alpha$;
\item Using the Polak-Ribiere variant of the Fletcher-Reeves algorithm \cite{NRC}, one of the
so-called conjugate gradient methods for maximisation of smooth functions, starting 
from each of these matrices, in order to find a candidate matrix for $A_l^*(\lambda)$, which is an approximate maximiser of $m_3(\cdot,\l)$.  
\end{itemize}
We choose $\alpha = 5$, the justification for which is as follows.    Five thousand matrices were generated and those which gave the 15 largest values of $m_3(\l,\cdot)$ were saved as the computation proceeded.  Of these, none had an element whose modulus exceeded $3.1$, hence reassuring us that the choice $\alpha = 5$ is `safe' for $\l$ in the range $[1,2]$.

Since the algorithm is iterative, a stopping condition is required,
and this is that
\begin{equation}
|m^{(i+1)} - m^{(i)}| \leq
\frac{\epsilon}{2}\left(|m^{(i+1)}| + |m^{(i)}|\right),
\label{stop}
\end{equation}
where $m^{(i)}$ is the value of $m_3(A_i, \lambda)$ at the $i$-th iteration
and $\epsilon = 10^{-9}$.

Various data are saved as the computation progresses, including the current maximising
matrix, which is the latest approximation to $A_l^*(\lambda)$. For half the
simulations, the initial random matrix is symmetric,
and for the other half it is not --- we do not know, \textit{a priori}, whether
$A_l^*(\lambda)$ will be symmetric or not.  The numerics strongly indicated that $A_3^*(\l)$ will indeed be symmetric, at least for $\l$ in the range $[1,2]$.

All computations were carried out using 40 significant figures.  Algorithm A leads the approximation $M_3(\l,\infty)\approx \nu_1^A/\l$, where $\nu_1^A=0.4501$ ---see  Table~\ref{modelM}; the algorithm also produces the approximation to $c_1(\l)$ shown in Figure 2 and summarised in Table \ref{modelC0} below.  

\begin{table}[ht]
\centering
\begin{tabular}{lcccccc}\hline
$\lambda$  & 1.01  & 1.1  & 1.2  & 1.3  & 1.4  & 1.5\\ \hline
\rb$M_3(\lambda,\infty)$ & 0.44566175 & 0.40919852 & 0.37509864 & 0.34624489 & 0.32151312 & 0.30007890\\
$\nu_1^A/\lambda$ & 0.44566173& 0.40919850& 0.37509862& 0.34624488& 0.32151311& 0.30007890\\ \hline\hline
$\lambda$  && 1.6  & 1.7  & 1.8  & 1.9 & 2.0\\\hline
\rb$M_3(\lambda,\infty)$ && 0.28132398 & 0.26477551 & 0.25006575 & 0.23690440 & 0.22505900\\
$\nu_1^A/\lambda$ && 0.28132397& 0.26477550& 0.25006575& 0.23690439& 0.22505917\\ \hline
\end{tabular}
\caption{Computed values of $M_3(\lambda,\infty)$ compared
with the approximation $\nu_1^A/\lambda$, for various values of $\lambda$.
Least squares was used to find $\nu_1^A$. The largest absolute deviation, $1.72\times
10^{-7}$, occurs at $\lambda = 2$.}
\label{modelM}
\end{table}

\begin{table}[ht]
\centering
\begin{tabular}{lcccccc}\hline
$\lambda$  & 1.01  & 1.1  & 1.2  & 1.3  & 1.4  & 1.5\\ \hline
$|A^*_3-\lambda\mathbf{1}|$ & 1.86212 & 2.02791 & 2.21231 & 2.39673 & 2.58098 & 2.76509\\
$c_1(\lambda)$ & 1.86240& 2.02820& 2.21242& 2.39664& 2.58087& 2.76509\\ \hline\hline
$\lambda$  && 1.6  & 1.7  & 1.8  & 1.9 & 2.0\\\hline
$|A^*_3-\lambda\mathbf{1}|$ && 2.94943 & 3.13421 & 3.31866 & 3.50261 & 3.68432\\
$c_1(\lambda)$ && 2.94931& 3.13353& 3.31775& 3.50197& 3.68619\\ \hline
\end{tabular}
\caption{Computed values of $|A^*_3-\lambda\mathbf{1}|$ compared
with the approximation $\nu_2 + \nu_3\lambda$, for various values of $\lambda$.}
\label{modelC0}
\end{table}

\subsection{Algorithm B: pointwise supremum}

This is based on a different idea, although a Monte Carlo approach
it is still at its heart.    We start by fixing an interval for $\lambda$, $\Lambda = [\lambda_-, \lambda_+]$, which is not
necessarily $[1, 2]$ --- the computation time is, at one level, independent of the interval. We
then define $N_p$ equally-spaced points in $\Lambda$, these points being
$\lambda_i = \lambda_- + i\delta\lambda$ with $\delta\lambda = (\lambda_+ -\lambda_-)/N_p$
and $i = 0, \ldots, N_p$.

Next, as before, a large number, $N$, of random matrices $A_i$ are generated. As can
be seen from its definition,
\begin{align*} m_l(A_i, \lambda) = \frac{\left|a_1 + a_2\lambda\right|}{(b_1 + b_2\lambda + b_3\lambda^2)^{l/2}}\end{align*}
where the coefficients $a_1,\ldots b_3$ are functions of the elements of $A_i$ that we compute
numerically. We define $f_{l, i}(\lambda) := m_l(A_i, \lambda)$, and clearly, once the coefficients
have been computed, $f_{l, i}(\lambda)$ can easily be found for any $\lambda$.
 
We then compute
$$F_{l, j} = \sup\left\{f_{l,i}(\lambda_j), i = 1,\ldots, N\right\}$$
for $j = 0,\ldots, N_p$; $F_{l, j}$ is then a discrete approximation to 
$M_l(\lambda_j,\infty)$.   The convergence to $M_3(\l,\infty)$ is quite slow, but nonetheless,
choosing $N$ large enough gives reasonable agreement with results
produced by Algorithm A, thereby providing an independent check.  Compare Fig. 1 with Fig. 2 below.

\begin{figure}[h]
\begin{center} 
\includegraphics[width=5.9in]{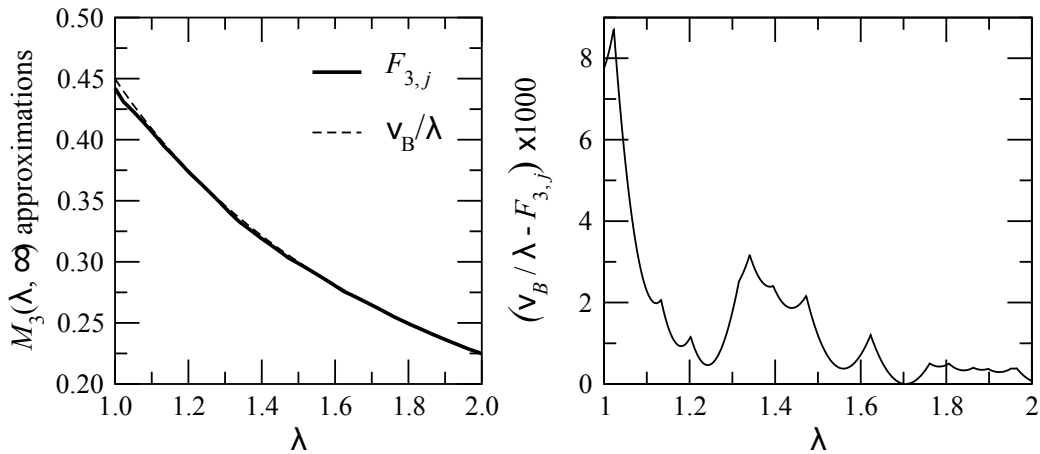}
\caption{Left: approximation from Algorithm B to $M_3(\lambda,\infty)$ and, for comparison,
$\nu_1^B/\lambda$, where $\nu_1^B = 0.4492$. Right: the difference between the
two curves in the left-hand figure.}
\end{center}
\label{G+algB}
\end{figure}

\subsection{Calculating $M^{-}_3(\l,\infty)$} \label{a3}

Using Algorithm A, the methodology is the same as for $M_3(\lambda,\infty)$, with the same number of random matrices generated, whose elements have
the same bounds. The investigations lead us to conjecture that
$$M^-_3(\lambda,\infty) \approx \nu_1^{A,-}/\lambda,$$ where $\nu_1^{A,-} \approx 0.1923$.  See Figure~\ref{G-algA}. 

\begin{figure}[h]
\centering 
\includegraphics*[width=3in]{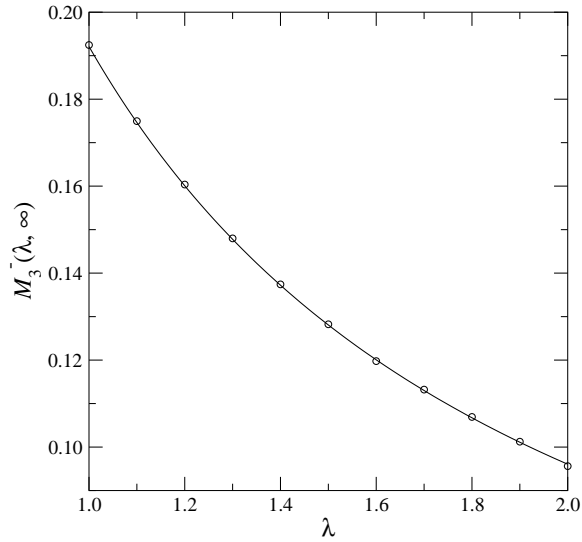}
\caption{The result of using Algorithm A to produce a plot of
$M^-_3(\lambda,\infty)$ versus $\lambda$, for $\lambda = 1.0, 1.1, \ldots, 2.0$, circles, where the
data has been obtained from the maximum of several computations. The
curve shows $\nu_1^{A,-}/\lambda$ for $\nu_1^{A,-} = 0.1923$.}
\label{G-algA}
\end{figure} 

\begin{figure}[htb!]
\centering 
\includegraphics*[width=3in]{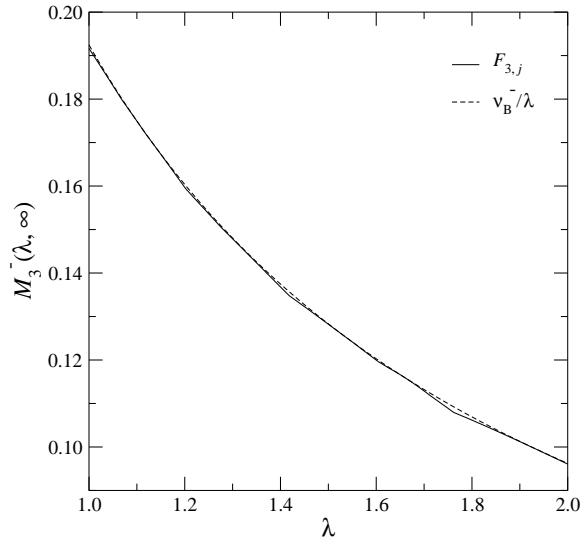}
\caption{Algorithm B used to estimate $M^-_3(\lambda,\infty)$ versus $\lambda$,
(continuous line), compared with the model $\nu_1^{B,-}/\lambda$ with $\nu_1^{B,-} =0.1925$ (dashed line).}
\label{G-algB}
\end{figure}

Recall that in the case of $M_3(\l,\infty)$ it was possible to compute and then model the quantity $c_1(\l)$ accurately on the interval $1 \leq \l \leq 2$ using an affine function of $\l$.  The same cannot be said of the corresponding quantity $c^{-}_1(\l)$, and indeed this seems to behave somewhat erratically as a function of $\l$.    Thus the analysis leading up to
Proposition \ref{identifyc*} does not apply, and hence the caveat regarding the substitution of $M^{-}_3(\l,\infty)\approx \nu_1'/\l$ promised in Remark \ref{silvermaple}.

\end{document}